\documentclass[12pt]{amsart}
\usepackage{amssymb}
\usepackage[all]{xy}

\input xy
\xyoption{all}

\newtheorem{lemma}{Lemma}[section]
\newtheorem{corollary}[lemma]{Corollary}
\newtheorem{theorem}[lemma]{Theorem}
\newtheorem{proposition}[lemma]{Proposition}

\theoremstyle{definition}
\newtheorem{remark}[lemma]{Remark}
\newtheorem{definition}[lemma]{Definition}

\newtheorem{example}[lemma]{Example}

\renewcommand{\L}{\Lambda}

\def\N{{\mathbb N}}
\def\l{{\lambda}}

\begin{document}
\title{The cycline subalgebra of a Kumjian-Pask algebra}

\author{Lisa Orloff Clark}
\address{Department of Mathematics and Statistics \\
University of Otago \\
P.O. Box: 56, Dunedin 9054, New Zealand}
\email{lclark@maths.otago.ac.nz}

\author{Crist\'{o}bal Gil Canto}
\address{Departamento de \'{A}lgebra, Geometr\'{\i}a y Topolog\'{\i}a\\
Universidad de M\'{a}laga\\
29071, M\'{a}laga, Spain}
\email{cristogilcanto@gmail.com / cgilc@uma.es}

\author{Alireza Nasr-Isfahani}
\address{Department of Mathematics\\
University of Isfahan\\
P.O. Box: 81746-73441, Isfahan, Iran\\ and School of Mathematics, Institute for Research in Fundamental Sciences (IPM), P.O. Box: 19395-5746, Tehran,
Iran}
\email{nasr$_{-}$a@sci.ui.ac.ir / nasr@ipm.ir}



\begin{abstract}
Let $\L$ be a row-finite higher-rank graph with no sources.
We identify a maximal commutative subalgebra
$\mathcal{M}$ inside the Kumjian-Pask algebra ${\rm KP}_R(\L)$.
We also prove a generalized Cuntz-Krieger uniqueness
theorem  for Kumjian-Pask algebras which says that
a representation of ${\rm KP}_R(\L)$ is injective if and only if
it is injective on $\mathcal{M}$.
\end{abstract}

\maketitle


\section{Introduction}

The Leavitt path algebra of a directed graph over a field is a specific type of path algebra associated to a
graph modulo some relations. Leavitt path algebras were introduced in \cite{AA1} and \cite{AMP}, and they are the purely algebraic version of Cuntz-Krieger graph $C^*$-algebras; on the other hand, they generalize the algebras without invariant basis number studied by Leavitt in \cite{L}. The relationship between the algebraic and analytic theories has been mutually beneficial.  Both families of algebras have proved to be rich sources of interesting examples and have
attracted interest from a broad range of researchers. In this paper we study analogues of Leavitt path algebras associated to higher-rank graphs; these
algebras are called Kumjian-Pask algebras. Concretely we extend the results given in \cite{GN} to Kumjian-Pask algebras.

Kumjian and Pask first introduced the notion of a higher-rank graph or $k$-graph $\L$ (in which paths have a $k$-dimensional
degree and a $1$-graph reduces to a directed graph) and the associated $C^*$-algebras $C^*(\L)$ in \cite{KP}.
These $C^*$-algebras provide a visualisable model for higher-rank versions of the Cuntz-Krieger algebras studied by Robertson and Steger
in \cite{RS}. The \emph{Kumjian-Pask algebra} ${\rm KP}_R(\L)$, defined and studied in \cite{ACaHR}, is an
algebraic version of $C^*(\L)$.
Kumjian-Pask algebras have a universal property based on a family of generators
satisfying suitable relations.   The study of
basic ideals and simplicity of ${\rm KP}_R(\L)$ is done in \cite{ACaHR}, the socle and semisimplicity are considered in \cite{BaH} and the
center is analysed in \cite{BaH2}. Kumjian-Pask algebras for more general graphs are considered in  \cite{CFaH} and \cite{CP}.

A central topic in $k$-graphs algebras
is to determine when a given homomorphism from ${\rm KP}_R(\L)$ (or $C^*(\L)$ in the analytic case) is injective;
this is the content of the so-called uniqueness theorems.
In \cite{ACaHR} a graded-uniqueness theorem and a Cuntz-Krieger uniqueness theorem  are proved
for ${\rm KP}_R(\L)$. Both require some conditions: the first one considers only $\mathbb{Z}^k$-graded homomorphisms, while the second requires the extra
hypothesis on the graph, that is, $\L$ is `aperiodic'.

In \cite{BNR},  a more general version of the
Cuntz-Krieger uniqueness theorem is proved in the $C^*$-algebraic setting that has no additional
hypotheses on the homomorphism or the graph.
Here we translate to Kumjian-Pask algebras the analytic result given in \cite{BNR}:
a representation of ${\rm KP}_R(\L)$ is injective if and only if it is injective on a
distinguished subalgebra, called the \emph{cycline subalgebra}.

At the same time we prove a more general version of the main results of \cite{GN} in the context of $1$-graphs. In \cite{GN} the second and third named authors prove a uniqueness theorem for Leavitt path algebras which establishes that the injectivity of a representation depends only on its injectivity on a certain commutative subalgebra \cite[Theorem 5.2]{GN}. Note that these results are not corollaries of the ones we obtain in this paper, since in \cite{GN} arbitrary graphs are considered and here we suppose that $k$-graphs are row-finite with no sources.

The paper is organized as follows. We begin with a section where we give
the background material, including the definition of ${\rm KP}_R(\L)$
and some basic properties. In Section 3 we establish some properties of the diagonal. In Section 4 we study the
cycline subalgebra.
Analogous to the
definition given in \cite{BNR},
the cycline subalgebra $\mathcal{M}$  is generated by elements of the
form $s_\alpha s_{\beta^*}$ where $(\alpha,\beta) \in \L \times \L$ is a \emph{cycline pair}
(Proposition \ref{Cyclinepairs}). In Theorem \ref{Mismaximalcommutative}, we prove that
$\mathcal{M}$ is a maximal commutative subalgebra inside ${\rm KP}_R(\L)$.

Finally in Section 5 we give our main result
Theorem \ref{UniquenessTh}: for $\Lambda$ a row-finite $k$-graph without sources,
$\Phi: {\rm KP}_R(\Lambda) \rightarrow A$ is an injective ring
homomorphism if and only if $\Phi |_{\mathcal{M}}$ is injective.


\section{Preliminaries}

First we give some necessary background which will be used later and we fix our notation.
Let $k$ be a positive integer. We consider the additive semigroup $\mathbb{N}^k$ as a category
with one object. We say a countable category $\Lambda = (\Lambda^0,\Lambda, r, s)$ with objects
$\Lambda^0$, morphisms $\Lambda$, range map $r$ and source map $s$, is a \emph{$k$-graph} if there
exists a functor $d : \Lambda \rightarrow  \mathbb{N}^k$, called the \emph{degree map}, satisfying
the \emph{unique factorization property}: if $d(\lambda) = m + n$ for some $m, n \in  \mathbb{N}^k$,
then there exist unique $\mu, \nu \in \mathbb{N}^k$ such that $r(\nu ) = s(\mu)$ and $d(\mu) = m$, $d(\nu ) = n$
with $\lambda = \mu\nu$. Since we think of $\Lambda$ as a generalized graph, we call $\lambda \in \Lambda$ a \emph{path}
in $\Lambda$ and $v \in \Lambda^0$ a \emph{vertex}.

For $n \in \mathbb{N}^k$ define $\Lambda^n = d^{-1}(\{n\})$ and call the elements $\lambda$ of $\Lambda^n$ \emph{paths of degree} $n$; by the factorization
property we identify $\Lambda^0$ as the paths of degree $0$ (or the set of vertices). For any $v \in \Lambda^0$ and $X \subseteq \Lambda$ we denote $vX =
\{ \lambda \in X \ | \ r(\lambda) = v\}$ and $Xv=\{\lambda \in X \ | \ s(\lambda)=v\}$. A $k$-graph $\Lambda$ has \emph{no sources} if for all $v \in
\Lambda^0$ and $n \in \mathbb{N}^k$ the set $v\Lambda^n$ is nonempty; $\Lambda$ is \emph{row-finite} if for all $v \in \Lambda^0$ and $n \in \mathbb{N}^k$
the set $v\Lambda^n$ is finite. In this paper we are concerned only with row-finite $k$-graphs without sources.

\begin{example} Let $E=(E^0,E^1,r,s)$ be a directed graph. Then the \emph{path category} $\Lambda_E$ has object set $E^0$, and the morphisms in $\Lambda_E$
from $v\in E^0$ to $w\in E^0$ are finite paths $\mu$ with $s(\mu)=v$ and $r(\mu)=w$;  composition is defined by concatenation, and the identity morphisms
obtained by viewing the vertices as paths of length $0$. With the degree functor $d: \Lambda_E \rightarrow \mathbb{N}$ as the length function, the path
category $(\Lambda_E,d)$ is a $1$-graph. Note that this requires us to use the convention where a path is a sequence of edges $e_1 \ldots e_n$ such that $s(e_i)=r(e_{i+1})$.
\end{example}

For $m,n \in \mathbb{N}^k$, $m\leq n$ means $m_i \leq n_i$ for all $1\leq i \leq k$ and $m \vee n$ denotes the pointwise maximum.

\begin{example} Let $\Omega_k^0 =\mathbb{N}^k$, $\Omega_k =\{(p,q)\in \mathbb{N}^k\times\mathbb{N}^k : p\leq q\}$, define $r,s:\Omega_k\to \Omega_k^0$ by
$r(p,q) =p$ and $s(p,q) =q$, define composition by $(p,q)(q,r)=(p,r)$, and define $d:\Omega_k\to\mathbb{N}^k$ by $d(p,q)=q-p$. Then
$\Omega_k=(\Omega_k,r,s,d)$ is a $k$-graph.

\end{example}

An \emph{infinite path} in $\Lambda$ is a
degree-preserving functor $x:\Omega_k\to \Lambda$. We denote the set of all
infinite paths by $\Lambda^\infty$. We denote $x(0,0)$ by $r(x)$ and refer to this vertex as the \emph{range} of $x$.

For $\mu \in \Lambda$ define the \emph{cylinder set} \[Z(\mu) :=\{ x \in \Lambda^{\infty} \ | \ x(0,d(\mu))= \mu\}.\]
For $\alpha \in \Lambda$ and $x \in Z(s(\alpha))$, define $\alpha x$ to be the unique infinite
path such that for any $n \geq d(\alpha)$, we have that $(\alpha x)(0,n)=\alpha(x(0,n-d(\alpha)))$.
We denote $Z(\alpha)$ by $\alpha\Lambda^{\infty}$ when $\alpha$ is a vertex.  The collection
of cylinder sets is a basis of compact sets for a Hausdorff topology on $\L^{\infty}$ (see \cite[Proposition 2.8]{KP}).

For $p \in \mathbb{N}^k$ we define a map $\sigma^p: \Lambda^{\infty} \rightarrow \Lambda^{\infty}$
by $\sigma^p(x)(m,n)=x(m+p,n+p)$ for every $x \in \Lambda^\infty$, $(m,n) \in \Omega_k$.
Note that by unique factorization $x=x(0,p)\sigma^p(x)$.

We say that a path $x \in \Lambda^\infty$ is
\emph{periodic} if there exists $p \neq q \in \mathbb{N}^k$ such that
$\sigma^p(x)=\sigma^q(x)$.  That is, for all $m,n \in \mathbb{N}^k$ $x(m+p,n+p)=x(m+q,n+q)$.
If $x$ is not periodic, we say $x$ is \emph{aperiodic}.
A row-finite $k$-graph with no sources is called \emph{aperiodic} if for every vertex $v \in \Lambda^0$,
there exists an aperiodic path in $v\Lambda^\infty$.

If $\L$ is a $k$-graph, we let $\L^{\not=0}:=\{\l\in\L:d(\l)\not=0\}$.
For each $\l\in\L^{\not=0}$ we introduce a \emph{ghost path}
$\l^*$ and for $v\in \L^0$ define $v^*:=v$. We write $G(\Lambda)$
for the set of ghost paths, or $G(\L^{\not=0})$ if we wish to exclude
vertices. We define $d$, $r$ and $s$ on $G(\L)$ by $d(\lambda^*) = -d(\lambda)$,
$r(\lambda^*) = s(\lambda)$, $s(\lambda^*) = r(\lambda)$; we then define composition
on $G(\L)$ by setting $\lambda^*\mu^*= (\mu\lambda)^*$ for $\lambda, \mu\in\L^{\neq 0}$
with $r(\mu^*) = s(\lambda^*)$. The factorization property of $\Lambda$ induces a similar
factorization property on $G(\L)$.

\begin{definition}\label{DefKumjianPask}
Let $\L$ be a row-finite $k$-graph without sources and let $R$ be a commutative ring with $1$. A \emph{Kumjian-Pask $\L$-family} $(P,S)$
in an $R$-algebra $A$ consists of two functions $P:\L^0\to A$ and $S:\L^{\not=0}\cup G(\L^{\not=0})\to A$ such that:
\begin{enumerate}
\item[(KP1)] $\{P_v:v\in \L^0\}$ is a family of mutually orthogonal idempotents;
\item[(KP2)] for all $\lambda, \mu\in\L^{\neq 0}$ with $r(\mu) = s(\lambda)$, we have
$$ S_{\lambda}S_{\mu} = S_{\lambda\mu}, \; S_{\mu^*}S_{\lambda^*} = S_{(\lambda\mu)^*}, \;
 P_{r(\lambda)}S_{\lambda} = S_{\lambda} = S_{\lambda}P_{s(\lambda)}, \;
  P_{s(\lambda)}S_{\lambda^*} = S_{\lambda^*} = S_{\lambda^*}P_{r(\lambda)};$$
\item[(KP3)] for all $\lambda, \mu \in\L^{\neq 0}$ with $d(\lambda) = d(\mu)$, we have $S_{\lambda^*}S_{\mu} = \delta_{\lambda,\mu}P_{s(\lambda)}$; and
\item[(KP4)] for all $v\in\L^0$ and all $n\in {\mathbb N}^k\setminus \{0\}$, we have $P_v = \sum_{\lambda\in v\L^n} S_{\lambda}S_{\lambda^*}$.
\end{enumerate}
\end{definition}

In \cite[Theorem 3.4]{ACaHR} it is proved that there is an $R$-algebra ${\rm KP}_R(\L)$ generated by a Kumjian-Pask $\L$-family $(p,s)$ with the following
universal property: whenever $(Q,T)$ is a Kumjian-Pask $\L$-family in an $R$-algebra $A$, there is a unique $R$-algebra homomorphism $\pi_{Q,T}: {\rm
KP}_R(\L) \rightarrow A$ such that $\pi_{Q,T}(p_v)=Q_v$, $\pi_{Q,T}(s_\l)=T_\l$, $\pi_{Q,T}(s_{\mu^*})=T_{\mu^*}$ for $v \in \L^0$ and $\l,\mu \in \L^{\neq
0}$. We call ${\rm KP}_R(\L)$ the \emph{Kumjian-Pask algebra} of $\L$ and the generating family $(p,s)$ the \emph{universal Kumjian-Pask family}.

With the convention that $s_v = p_v$ and $s_{v^*}=p_v$ it follows from the \emph{Kumjian-Pask relations} (KP1)-(KP4) that
$${\rm KP}_R(\L)=\text{span}_R\{s_\mu s_{\nu^*}: \mu,\nu \in \L \text{ with } s(\mu)=s(\nu)\},$$
and $s_\mu s_{\nu^*} \neq 0$ in ${\rm KP}_R(\L)$ if $s(\mu)=s(\nu)$.
For every nonzero $a \in {\rm KP}_R(\L)$ and $n \in \N^k$ there exist $m \geq n$ and
a finite subset $F \subseteq \L \times \L^m$ such that $s(\alpha)=s(\beta)$ for
all $(\alpha,\beta) \in F$ and
\[a = \sum_{(\alpha,\beta) \in F} r_{\alpha,\beta}s_{\alpha}s_{\beta^*}\]
with $r_{\alpha,\beta} \in R \setminus \{0\}$. In this situation,
we say that $a$ is written in \emph{normal form} \cite[Lemma 4.2]{ACaHR}.
We can define an $R$-linear involution $a \mapsto a^*$ on ${\rm KP}_R(\L)$ as follows: if $a = \sum_{(\alpha,\beta) \in F}
r_{\alpha,\beta}s_{\alpha}s_{\beta^*}$ then $a^* = \sum_{(\alpha,\beta) \in F} r_{\alpha,\beta}s_{\beta}s_{\alpha^*}$.

It is proved in \cite[Theorem 3.4]{ACaHR} that ${\rm KP}_R(\L)$ is graded by $\mathbb{Z}^k$ such that for each $n \in \mathbb{Z}^k$,
$${\rm KP}_R(\Lambda)_n=\text{span}_R\{s_{\lambda}s_{\mu^*}:\lambda,\mu \in \Lambda \text{ and }d(\lambda)-d(\mu)=n\}.$$
\cite[Theorem 3.4]{ACaHR} also says that $rp_v \neq 0$ for $v \in \L^0$ and $r \in R \setminus \{0\}$.


\section{The diagonal subalgebra}

In this section we establish some properties of  the \emph{diagonal subalgebra}  $\mathcal{D}$ of the
Kumjian-Pask algebra ${\rm KP}_R(\L)$. Define
\[\mathcal{D} := \ \langle s_\mu s_{\mu^*}: \mu \in \L \rangle ;\]
that is, $\mathcal{D}$ is the $R$-subalgebra of
${\rm KP}_R(\L)$ generated by the set $\{s_\mu s_{\mu^*}: \mu \in \L\}$.
Observe that for $ v \in \L^0$, $s_v s_{v^*}=p_v$.

\begin{lemma} The diagonal subalgebra $\mathcal{D}$ is commutative.
\end{lemma}

\begin{proof} Using the Kumjian-Pask relations it follows that for each $n$, $\{s_\mu s_{\mu^*}: \mu \in \L^n\}$ is a set of mutually orthogonal
idempotents (see \cite[Remarks 3.2]{ACaHR}). Now if we consider $\l, \mu \in \L$ then by \cite[Lemma 3.3]{ACaHR}, for each $q\geq d(\lambda)\vee d(\mu)$ we
have that
$$ s_{\lambda^*}s_{\mu} = \sum_{d(\lambda\alpha)=q,\;\lambda\alpha = \mu\beta}s_{\alpha}s_{\beta^*}.$$
Then $(s_\l s_{\lambda^*})(s_{\mu} s_{\mu^*}) = (s_\mu s_{\mu^*})(s_\l s_{\l^*})$ because:
\begin{align*}
(s_\l s_{\lambda^*})(s_{\mu} s_{\mu^*}) &= s_\l (s_{\lambda^*}s_{\mu}) s_{\mu^*} \\
&=  s_\l (\sum_{d(\lambda\alpha)=q,\;\lambda\alpha= \mu\beta}s_{\alpha}s_{\beta^*}) s_{\mu^*} \\
&= \sum_{d(\lambda\alpha)=q,\;\lambda\alpha = \mu\beta}s_\l s_{\alpha}s_{\beta^*}s_{\mu^*}\\
&= \sum_{d(\lambda\alpha)=q,\;\lambda\alpha = \mu\beta}s_{\l\alpha}s_{{(\mu\beta)}^*} \\
&=\sum_{d(\mu\beta)=q,\;\mu\beta=\lambda\alpha}s_{\mu\beta}s_{{(\l\alpha)}^*}\\
&=\sum_{d(\mu\beta)=q,\;\mu\beta=\lambda\alpha}s_{\mu}s_{\beta}s_{\alpha^*}s_{\l^*}\\
\end{align*}
\begin{align*}
&= s_{\mu}(\sum_{d(\mu\beta)=q,\;\mu\beta=\lambda\alpha}s_{\beta}s_{\alpha^*})s_{\l^*}\\
&= s_\mu (s_{\mu^*}s_\l) s_{\l^*} \\
&= (s_\mu s_{\mu^*})(s_\l s_{\l^*}).
\end{align*}
So $\mathcal{D}$ is commutative.
\end{proof}

\begin{remark}\label{remark:Z} Observe that
$$s_{\alpha\gamma}s_{(\alpha\gamma)^*}s_\alpha s_{\beta^*} s_{\beta\eta}s_{(\beta\eta)^*}= s_{\alpha\gamma}s_{\gamma^*} s_\eta s_{(\beta\eta)^*}.$$
In particular if $d(\gamma)=d(\eta)$, then
\begin{align*}
s_{\alpha\gamma}s_{(\alpha\gamma)^*}s_\alpha s_{\beta^*} s_{\beta\eta}s_{(\beta\eta)^*} & = \begin{cases}
s_{\alpha\gamma}s_{(\beta\gamma)^*} & \text{if }\ \gamma=\eta,\\
0 & \text{otherwise. }\  \end{cases}
\end{align*}
\end{remark}

We may view elements of the diagonal
$\mathcal{D}$ as functions from $\L^{\infty}$ to $R$ in the following way.
For $\mu \in \Lambda$ let $1_{Z(\mu)}$ denote the characteristic function association to $Z(\mu)$.  That is,
$1_{Z(\mu)} : \Lambda^{\infty} \rightarrow R$ such that
\begin{align*}
1_{Z(\mu)}(x)  & = \begin{cases}
1 & \text{if }\ x \in Z(\mu),\\
0 & \text{otherwise. }\  \end{cases}
\end{align*}
Let
\[A_{\mathcal{D}}:= {\rm span}_R\{1_{Z(\mu)}: \ \mu \in \L \}.\]
Then $A_{\mathcal{D}}$ as an $R$-algebra with addition and scalar multiplication defined pointwise and multiplication
of the generators is given by
$$1_{Z(\mu)} \cdot 1_{Z(\nu)} = 1_{Z(\mu) \cap Z(\nu)}.$$

\begin{remark}
In fact $A_{\mathcal{D}}$ is the
diagonal of the \emph{Steinberg algebra} associated to $\L$ (see Section 5 of \cite{CP} for more details).
\end{remark}
The following lemma follows from \cite[Proposition 5.4]{CP}:
\begin{lemma}\label{DiagonalIsomorphismLemma} The map $\pi: \mathcal{D} \rightarrow A_{\mathcal{D}}$ such that $\pi(s_\mu s_{\mu^*})= 1_{Z(\mu)}$ is an
isomorphism.
\end{lemma}

\begin{lemma}\label{Diagonal1Prop} Let $\mu, \nu \in \L$ with $s(\mu)=s(\nu)$ and $r \in R$ such that
$r s_\mu s_{\mu^*} = r s_\nu s_{\nu^*}$. Then $s_\mu s_{\mu^*} = s_\nu s_{\nu^*}$.
\end{lemma}

\begin{proof} Let $\pi: \mathcal{D} \rightarrow A_{\mathcal{D}}$. Then $\pi(rs_\mu s_{\mu^*})= \pi(rs_\nu
s_{\nu^*})$. Therefore we have $r 1_{Z(\mu)} = r 1_{Z(\nu)}$, which implies $Z(\mu) = Z(\nu)$. This means $s_\mu s_{\mu^*} = s_\nu s_{\nu^*}$ because the map $\pi$ of Lemma \ref{DiagonalIsomorphismLemma} is injective.
\end{proof}

\section{The cycline subalgebra}

We study a special class of generators for ${\rm KP}_R(\L)$ which will be used in the construction of a
distinguished subalgebra, called \emph{the cycline subalgebra} of ${\rm KP}_R(\L)$. First an element $a \in {\rm KP}_R(\L)$ is said to be \emph{normal} if $aa^*=a^*a$.
The proof of the following proposition follows exactly as the one given in
\cite[Proposition 4.1]{BNR}.  Note that in \cite[Proposition~4.1]{BNR}, they write
$P_{\alpha}$ for the projection $s_\alpha s_{\alpha^*}$.

\begin{proposition}\label{Cyclinepairs}
Let $\L$ be a row-finite
$k$-graph with no sources and $R$ be a commutative ring with $1$.
Then for $(\alpha,\beta) \in \L \times \L$
with $s(\alpha)=s(\beta)$, the following conditions are equivalent:
\begin{enumerate}
\item \label{it1:cyclinepairs} $s_{\alpha\gamma}s_{(\alpha\gamma)^*}=s_{\beta\gamma}s_{(\beta\gamma)^*}$ for all $\gamma \in s(\alpha)\L$;
\item \label{it2:cyclinepairs} $s_\alpha s_{\beta^*}$ is normal and commutes with $\mathcal{D}$;  and
\item \label{it3:cyclinepairs} $\alpha\gamma=\beta\gamma$ for all $\gamma \in s(\alpha)\L^{\infty}$.
\end{enumerate}
\end{proposition}

\begin{definition} A pair $(\alpha,\beta) \in \L \times \L$ with $s(\alpha)=s(\beta)$
satisfying the equivalent conditions of Proposition \ref{Cyclinepairs} is called a \emph{cycline pair}.
Define \[\mathcal{M} := \ \langle s_\alpha s_{\beta^*}: (\alpha,\beta) \text{ cycline}  \rangle; \]
that is, the $R$-subalgebra of ${\rm KP}_R(\L)$ generated by the cycline pairs.
We call $\mathcal{M}$ the \emph{cycline subalgebra} of the Kumjian-Pask algebra ${\rm KP}_R(\L)$.
\end{definition}

\begin{remark}\label{RemarkCycline} Observe that $(\alpha,\alpha)$ is cycline and hence $\mathcal{D} \subseteq \mathcal{M}$.
Also $(\alpha, \beta)$ is cycline if and only if $(\beta,\alpha)$ is. For $(\alpha, \beta), (\mu, \nu) \in \L \times \L$ and $n \in \N^k$, $n \geq d(\alpha) \vee d(\beta)$ we have
\begin{equation}\label{eq:cyc}s_\alpha s_{\beta^*} s_\mu s_{\nu^*} = \sum_{\gamma, \eta \in \L,
 \; \beta\gamma=\mu\eta, \; d(\beta\gamma)=n}s_{\alpha\gamma}s_{(\nu\eta)^*}.\end{equation}
If $(\alpha, \beta), (\mu, \nu) $ are cycline and $\lambda \in s(\gamma)\L$, then
$$s_{\nu\eta\lambda}s_{(\nu\eta\lambda)^*}=s_{\mu\eta\lambda}s_{(\mu\eta\lambda)^*}=s_{\beta\gamma\lambda}s_{(\beta\gamma\lambda)^*}=
s_{\alpha\gamma\lambda}s_{(\alpha\gamma\lambda)^*}.$$
This means that the pairs $(\alpha\gamma,\nu\eta)$ appearing in the right hand side of
\eqref{eq:cyc} are cycline too  by Proposition \ref{Cyclinepairs} \eqref{it1:cyclinepairs}.
\end{remark}

\begin{lemma}
\label{lem:Mcom}The cycline subalgebra $\mathcal{M}$ is commutative.
\end{lemma}

\begin{proof} Let $(\alpha,\beta)$ and $(\mu,\nu)$ be two cycline pairs. It suffices to show that
$$ (s_\alpha s_{\beta^*}) (s_\mu s_{\nu^*}) = (s_\mu s_{\nu^*}) (s_\alpha s_{\beta^*}).$$
Now
\begin{align*}
 s_\alpha s_{\beta^*} \  s_\mu s_{\nu^*} &= \sum_{\beta\gamma=\mu\eta}s_{\alpha\gamma}s_{(\nu\eta)^*} \ \\
& = \sum_{\beta\gamma=\mu\eta, \ \nu\eta=\alpha\gamma}s_{\alpha\gamma} \ (s_{(\alpha\gamma)^*}s_{\nu\eta}) \ s_{(\nu\eta)^*} \text{ by (KP3)} \\
& = \sum_{\beta\gamma=\mu\eta, \ \nu\eta=\alpha\gamma}(s_{\alpha\gamma}s_{(\alpha\gamma)^*}) \ (s_{\nu\eta}s_{(\nu\eta)^*}) \ \\
& = \sum_{\beta\gamma=\mu\eta, \ \nu\eta=\alpha\gamma}(s_{\beta\gamma}s_{(\beta\gamma)^*}) \ (s_{\mu\eta}s_{(\mu\eta)^*}) \ \text{ since $(\alpha,\beta)$
and $(\mu,\nu)$ are cycline} \\
& = \sum_{\beta\gamma=\mu\eta, \ \nu\eta=\alpha\gamma} (s_{\mu\eta}s_{(\mu\eta)^*}) \ (s_{\beta\gamma}s_{(\beta\gamma)^*}) \text{ since $\mathcal{D}$ is
commutative} \\
& = \sum_{\beta\gamma=\mu\eta, \ \nu\eta=\alpha\gamma} s_{\mu\eta} \ (s_{(\mu\eta)^*}s_{\beta\gamma}) \ s_{(\beta\gamma)^*}  \ \\
\end{align*}
\begin{align*}
& = \sum_{\nu\eta=\alpha\gamma} s_{\mu\eta} s_{(\beta\gamma)^*} \\
& = s_\mu s_{\nu^*} \ s_\alpha s_{\beta^*} \  \text{as desired.}\qedhere
\end{align*}
\end{proof}

\begin{remark} Let $E$ be a $1$-graph. In \cite[Example 4.9]{BNR}, it is seen that a standard generator $s_\alpha s_{\beta^*}$ is cycline if and only if
$\alpha=\beta$, $\alpha=\beta c$ or $\beta=\alpha c$ for some cycle $c \in s(\alpha)E s(\alpha)$ without entry. Thus the cycline subalgebra $\mathcal{M}$
coincides with \emph{the commutative core} $M_R(E)$ defined in \cite[Definition 3.15]{GN}.
\end{remark}

Define $\mathcal{D'}$ such that
\begin{equation}\label{Dprime}\mathcal{D'} := \{a \in {\rm KP}_R(\Lambda): ad=da \text{ for every } d \in \mathcal{D}\}.\end{equation}
Notice that we have the inclusions $\mathcal{D} \subseteq \mathcal{M} \subseteq \mathcal{D'}$, where the
second inclusion comes from Proposition~\ref{Cyclinepairs} \eqref{it2:cyclinepairs}.
In fact, $\mathcal{M}=\mathcal{D}$ if $\L$ is aperiodic (see Lemma \ref{lem:aperiodic}).
Our main goal in this section is to prove the following.

\begin{theorem}\label{Mismaximalcommutative}Let $\L$ be a row-finite
$k$-graph with no sources, $R$ be a commutative ring with $1$,
$\mathcal{M}$ be the cycline subalgebra and
$\mathcal{D'}$ be defined as in \eqref{Dprime}. Then
$\mathcal{M}=\mathcal{D'}$. In particular, $\mathcal{M}$ is a maximal commutative subalgebra of ${\rm KP}_R(\L)$.
\end{theorem}

To prove Theorem~\ref{Mismaximalcommutative}, we establish three lemmas.
Recall that  ${\rm KP}_R(\Lambda)$ is a $\mathbb{Z}^k$-graded algebra such that
for $n \in \mathbb{Z}^k$,
\[{\rm KP}_R(\Lambda)_n=\text{span}_R\{s_{\lambda}s_{\mu^*}:\lambda,\mu \in \Lambda \text{ and }d(\lambda)-d(\mu)=n\}.\]

\begin{lemma}\label{NormalLemma} Let $a \in \mathcal{D'}$ and suppose $a = \sum_{(\alpha,\beta) \in F} r_{\alpha,\beta}s_{\alpha}s_{\beta^*}$ is in normal
form. Then for any $(\alpha,\beta) \in F$, $r_{\alpha,\beta}s_{\alpha}s_{\beta^*} \in \mathcal{D'}$.
\end{lemma}

\begin{proof} First we claim that $a \in \mathcal{D'}$ if and only if $a_n \in \mathcal{D'}$ for all $n \in
\mathbb{Z}^k$, where $a_n$ is the homogeneous part of $a$ of degree $n$. To prove the claim note if for all $n \in \mathbb{Z}^k$, $a_n \in \mathcal{D'}$ then $a \in \mathcal{D'}$. Assume that $a \in \mathcal{D'}$. Then for any $d \in
\mathcal{D}$, $ad=da$. Since elements of $\mathcal{D}$ are homogeneous of degree zero then $a_n d = {(ad)}_n$. Then
$$a_n d = {(ad)}_n = {(da)}_n = d a_n,$$
and so $a_n \in \mathcal{D'}$ for every $n \in \mathbb{Z}^k$.

Thus to prove the lemma, we can assume that $a$ is homogeneous of degree $n$. Since all $\beta$ have the same degree, then all
$\alpha$ have the same degree. For every $(\gamma,\delta) \in F$, $s_{\gamma}s_{\gamma^*} \in \mathcal{D}$ and $s_{\delta}s_{\delta^*} \in \mathcal{D}$ so
then $s_{\gamma}s_{\gamma^*} a s_{\delta}s_{\delta^*} \in \mathcal{D'}$. By (KP3) we have, $s_{\gamma}s_{\gamma^*} a s_{\delta}s_{\delta^*} =
r_{\gamma,\delta} s_\gamma s_{\delta^*} \in \mathcal{D'}$.
 \end{proof}

In the following lemma we adopt some ideas from \cite[Lemma 3.2]{Y}.

\begin{lemma}\label{Diagonal2Lemma} Let $\mu, \nu \in \L$ with $s(\mu)=s(\nu)$ and $r \in R \setminus \{0\}$ such that $r s_\mu s_{\nu^*} \in
\mathcal{D'}$. Then $s_\mu s_{\mu^*} = s_\nu s_{\nu^*}$.
\end{lemma}

\begin{proof} First since $rs_\mu s_{\nu^*} \in \mathcal{D'}$ we have $r s_\mu s_{\mu^*} s_\mu s_{\nu^*} = r s_\mu s_{\nu^*} s_\mu s_{\mu^*}$, that is,  $r
s_\mu s_{\nu^*} =r s_\mu s_{\nu^*} s_\mu s_{\mu^*}$. Then multiplying both sides of this equation by $s_{\mu^*}$ on the left, we get $r s_{\mu^*} s_\mu
s_{\nu^*} = r s_{\mu^*} s_\mu s_{\nu^*} s_\mu s_{\mu^*}$, which means by (KP3) that $r s_{\nu^*}= r s_{\nu^*} s_\mu s_{\mu^*}$. Finally multiplying by
$s_{\nu}$  again on the left gives,
$$r s_\nu s_{\nu^*}=r s_\nu s_{\nu^*} s_\mu s_{\mu^*}.$$

Since $rs_\mu s_{\nu^*} \in \mathcal{D'}$, we have $r s_\nu s_{\mu^*} \in \mathcal{D'}$. Using the same argument we get
$$r s_\mu s_{\mu^*}= r s_\mu s_{\mu^*} s_\nu s_{\nu^*}.$$

 Now $s_\nu s_{\nu^*}, s_\mu s_{\mu^*} \in \mathcal{D}$ and $\mathcal{D}$ is commutative so
$$r s_\mu s_{\mu^*}= r s_\mu s_{\mu^*} s_\nu s_{\nu^*} = r s_\nu s_{\nu^*} s_\mu s_{\mu^*} = r s_\nu s_{\nu^*}.$$

That is, $r s_\mu s_{\mu^*} = r s_\nu s_{\nu^*}$. Finally by Lemma \ref{Diagonal1Prop} we obtain
$s_\mu s_{\mu^*} = s_\nu s_{\nu^*}$.
\end{proof}

\begin{lemma}\label{Diagonal3Lemma} Let $\mu, \nu \in \L$ with $s(\mu)=s(\nu)$ and $r \in R \setminus \{0\}$ such that $r s_\mu s_{\nu^*} \in
\mathcal{D'}$. Then $(\mu, \nu)$ is a cycline pair.
\end{lemma}

\begin{proof} Let $\gamma \in s(\mu)\L$. We show $s_{\mu\gamma}s_{(\mu\gamma)^*} = s_{\nu\gamma}s_{(\nu\gamma)^*}$. Because $r s_\mu s_{\nu^*} \in \mathcal{D'}$, $r s_\nu s_{\mu^*} \in \mathcal{D'}$. So, on the one hand, we have:
\begin{align*}
(r s_{\mu\gamma}s_{(\mu\gamma)^*} s_{\mu} s_{\nu^*})(s_{\mu\gamma}s_{(\mu\gamma)^*} s_{\mu} s_{\nu^*})^* & = (s_{\mu\gamma}s_{(\mu\gamma)^*} s_{\mu}
s_{\nu^*}) (r s_\nu s_{\mu^*})s_{\mu\gamma}s_{(\mu\gamma)^*}\\
& = (s_{\mu\gamma}s_{(\mu\gamma)^*} s_{\mu} s_{\nu^*}) s_{\mu\gamma}s_{(\mu\gamma)^*} (r s_\nu s_{\mu^*}) \\
&= s_{\mu\gamma}s_{(\mu\gamma)^*} (r s_{\mu} s_{\nu^*}) s_{\mu\gamma}s_{(\mu\gamma)^*} s_\nu s_{\mu^*}  \\
& = s_{\mu\gamma}s_{(\mu\gamma)^*} s_{\mu\gamma}s_{(\mu\gamma)^*} (r s_{\mu} s_{\nu^*}) s_\nu s_{\mu^*} \\
& = r s_{\mu\gamma}s_{(\mu\gamma)^*} s_{\mu} s_{\mu^*}.
\end{align*}

On the other hand,
\begin{align*}
(s_{\mu\gamma}s_{(\mu\gamma)^*} s_{\mu} s_{\nu^*})^*(r s_{\mu\gamma}s_{(\mu\gamma)^*} s_{\mu} s_{\nu^*}) & = (r s_\nu s_{\mu^*})
s_{\mu\gamma}s_{(\mu\gamma)^*} s_{\mu\gamma}s_{(\mu\gamma)^*} s_{\mu} s_{\nu^*} \\
& = s_{\mu\gamma}s_{(\mu\gamma)^*} (r s_{\nu} s_{\mu^*}) s_\mu s_{\nu^*} \\
&= r s_{\mu\gamma}s_{(\mu\gamma)^*} s_{\nu} s_{\nu^*}.
\end{align*}

But by Lemma \ref{Diagonal2Lemma}, $r s_\mu s_{\nu^*} \in \mathcal{D'}$ implies $s_\mu s_{\mu^*} = s_\nu s_{\nu^*}$. Therefore
\begin{equation}\label{eq:dl} (r s_{\mu\gamma}s_{(\mu\gamma)^*} s_{\mu} s_{\nu^*})(s_{\mu\gamma}s_{(\mu\gamma)^*} s_{\mu} s_{\nu^*})^* =
(s_{\mu\gamma}s_{(\mu\gamma)^*} s_{\mu} s_{\nu^*})^*(r s_{\mu\gamma}s_{(\mu\gamma)^*} s_{\mu} s_{\nu^*}).
\end{equation}

Now:
\begin{align*}
 r s_{\mu\gamma}s_{(\mu\gamma)^*} & = (r s_{\mu\gamma}s_{(\mu\gamma)^*} s_{\mu} s_{\nu^*})(s_{\mu\gamma}s_{(\mu\gamma)^*} s_{\mu} s_{\nu^*})^* \\
& = (s_{\mu\gamma}s_{(\mu\gamma)^*} s_{\mu} s_{\nu^*})^*(r s_{\mu\gamma}s_{(\mu\gamma)^*} s_{\mu} s_{\nu^*}) \text{ by \eqref{eq:dl}} \\
& = r s_{\nu\gamma}s_{(\nu\gamma)^*}.
\end{align*}
Hence $r s_{\mu\gamma}s_{(\mu\gamma)^*} = r s_{\nu\gamma}s_{(\nu\gamma)^*}$ and by Lemma \ref{Diagonal1Prop},
we obtain
$$s_{\mu\gamma}s_{(\mu\gamma)^*} = s_{\nu\gamma}s_{(\nu\gamma)^*}.$$
Thus by Proposition \ref{Cyclinepairs} \eqref{it1:cyclinepairs}, $(\mu,\nu)$ is a cycline pair.
\end{proof}

We are now ready to prove the main result of this section.
\begin{proof}[Proof of Theorem~\ref{Mismaximalcommutative}] To see that $\mathcal{M}=\mathcal{D'}$, we show
$\mathcal{D'} \subseteq \mathcal{M}$. Let $a \in \mathcal{D'}$.
Suppose $a$ is written in normal form
\[ a = \sum_{(\alpha,\beta) \in F} r_{\alpha,\beta}s_{\alpha}s_{\beta^*}\]
By Lemma \ref{NormalLemma}, for each $(\alpha,\beta) \in F$ we have $r_{\alpha,\beta}s_{\alpha}s_{\beta^*} \in \mathcal{D'}$. Now we apply Lemma
\ref{Diagonal3Lemma} to see that each $(\alpha, \beta)$ in $F$ is a cycline pair. Thus $a \in \mathcal{M}$.


In order to prove that $\mathcal{M}$ is a maximal commutative subalgebra
inside ${\rm KP}_R(\L)$, first recall that
$\mathcal{M}$ is commutative by Lemma~\ref{lem:Mcom}.  Now
consider $\mathcal{C}$ a commutative subalgebra of ${\rm KP}_R(\L)$ such that $\mathcal{M} \subseteq \mathcal{C}$. Since we have $\mathcal{D} \subseteq
\mathcal{M} \subseteq \mathcal{C}$ then in particular
$$\mathcal{C} \subseteq \{x \in {\rm KP}_R(\L): xd=dx \text{ for every } d \in \mathcal{D}\} = \mathcal{D'}.$$
Therefore $\mathcal{C} = \mathcal{M}$.
\end{proof}

The following corollary involving the center of a Kumjian-Pask algebra,  (studied in \cite{BaH2}) is immediate.

\begin{corollary}Let $\L$ be a row-finite
$k$-graph with no sources and $R$ be a commutative ring with $1$.   Then the center of the
Kumjian-Pask algebra
$$\mathcal{Z}({\rm KP}_R(\L)) := \{ a \in {\rm KP}_R(\Lambda): ax=xa \text{ for every } x \in {\rm KP}_R(\L) \}\subseteq \mathcal{M}.$$
\end{corollary}

\section{General uniqueness theorem for Kumjian-Pask algebras}

In this section we give a new uniqueness theorem for Kumjian-Pask algebras that says a homomorphism
on ${\rm KP}_R(\Lambda)$ is injective if and only if it is injective
on the cycline subalgebra $\mathcal{M}$. First we adapt some of the technical innovations from \cite{BNR} to our setting.

For any subset $U \subseteq X$ of a topological space $X$ we denote its
interior by ${\rm Int}(U)$ and its boundary by $\partial U$. The following definition appears
in \cite[Definition 5.2]{BNR}.  For a $k$-graph $\L$, define
\[\Sigma := \{(\alpha,\beta) \in \L \times \L \ | \ s(\alpha)=s(\beta), \alpha \neq \beta \}.\]
Then for any pair $(\alpha,\beta) \in \Sigma$ let
$$F_{\alpha,\beta}:=\{x \in \L^{\infty} \ | \ x \in Z(\alpha) \cap Z(\beta) \text{ and } \sigma^{d(\alpha)}(x)= \sigma^{d(\beta)}(x) \} .$$
We define the set of \emph{regular paths} in $\L^{\infty}$ to be
$$\mathfrak{T}_\L:=\L^{\infty} - \bigcup_{(\alpha,\beta) \in \Sigma} \partial F_{\alpha,\beta}.$$

\begin{remark}\label{RemarkRegularPaths} In this remark we recall some properties of $F_{\alpha,\beta}$ and $\mathfrak{T}_\L$, which are given in
\cite[Section 5]{BNR}.
\begin{itemize}
\item[(a)] We have $F_{\alpha,\beta} = F_{\beta,\alpha}$.
\item[(b)] We have that $F_{\alpha,\beta}$ is closed for all $(\alpha, \beta) \in \Sigma$. Indeed if $x_i \rightarrow x$ in $\L^{\infty}$ and $n \in
    \N^k$  then $\sigma^n(x_i) \rightarrow \sigma^n(x)$; in particular if $x_i \in F_{\alpha,\beta}$ for all $i$ then $x(0,d(\alpha))= \alpha$,
    $x(0,d(\beta))= \beta$ and $\sigma^{d(\alpha)}(x)=\sigma^{d(\beta)}(x)$, so $x \in F_{\alpha,\beta}$. Therefore, $\partial F_{\alpha,\beta} \subseteq
    F_{\alpha,\beta}$. Note that $\partial F_{\alpha,\beta}$ is also closed.
\item[(c)] The set $\mathfrak{T}_\L$ is dense in $\L^{\infty}$ (by the Baire Category Theorem).
So for every $v \in \L^0$ there exists an $x \in Z(v) \cap \mathfrak{T}_\L$.
\item[(d)] We have $\mathfrak{T}_\L \cap F_{\alpha,\beta} \subseteq {\rm Int}(F_{\alpha,\beta})$.
\item[(e)] The cylinder set $Z(\mu) \subseteq F_{\alpha,\beta}$ if and only if $Z(\nu\mu) \subseteq F_{\nu\alpha,\nu\beta}$.
\item[(f)] If $x \in \mathfrak{T}_\L$ and $\nu \in \L r(x)$, then $\nu x \in \mathfrak{T}_\L$.
\item[(g)] If $x \in \mathfrak{T}_\L$ and $n \in \N^k$, then $\sigma^n(x) \in \mathfrak{T}_\L$.
\end{itemize}
\end{remark}

\begin{remark} Let $E$ be a $1$-graph. An infinite path in $E$ is called \emph{essentially aperiodic} if it is either aperiodic or of the form $\mu
ccc\cdots$, where $c$ is a cycle without entry (see \cite[Definition 3.5 (B)]{GN} for more details). In \cite[Example 5.4]{BNR}, it is seen that
$\mathfrak{T}_E$ is the set of essentially aperiodic paths of $E$.
\end{remark}

The next lemma corresponds to \cite[Lemma 5.8]{BNR} and the proof translates exactly (so we omit it).

\begin{lemma}\label{ReductionLemma} For $(\alpha, \beta) \in \Sigma$ and $x \in \mathfrak{T}_\L$ we have the following:
\begin{itemize}
\item[(a)] If $x \notin F_{\alpha, \beta}$, then there exists $\mu, \nu \in \L$ such that $x \in Z(\mu) \cap Z(\nu)$ and
$$ s_\mu s_{\mu^*} s_{\alpha} s_{\beta^*} s_\nu s_{\nu^*} = 0.$$
\item[(b)] If $x \in F_{\alpha, \beta}$, then there exists $\gamma \in \L$ with $x \in Z(\alpha\gamma) \cap Z(\beta\gamma)$,
$$ s_{\alpha\gamma} s_{(\alpha\gamma)^*} s_{\alpha} s_{\beta^*} s_{\beta\gamma} s_{(\beta\gamma)^*} =  s_{\alpha\gamma} s_{(\beta\gamma)^*}, \text{ and }
(\alpha\gamma, \beta\gamma) \text{ cycline}.$$
\end{itemize}
\end{lemma}

\begin{theorem}\label{UniquenessTh}Let $\L$ be a row-finite
$k$-graph with no sources, $R$ be a commutative ring with $1$ and $\mathcal{M}$ be the cycline subalgebra of
${\rm KP}_R(\Lambda)$. If $\Phi: {\rm KP}_R(\Lambda) \rightarrow A$ is a ring homomorphism,
then $\Phi$ is injective if and only if $\Phi |_{\mathcal{M}}$ is injective.
\end{theorem}


\begin{proof} We show that $\Phi |_{\mathcal{M}}$ injective implies $\Phi$ injective.
By way of contradiction suppose we have $0 \neq a \in {\rm KP}_R(\L)$ such that $a \in {\rm Ker}\ \Phi$.
Applying \cite[Lemma 2.3 (i)]{BaH} (and writing the sum involved in normal form) we can find
$(\delta,\epsilon) \in \L \times \L$ such that
$$0 \neq  s_{\delta^*} a s_\epsilon = r_{\delta,\epsilon} p_{s(\delta)} + \sum_{(\alpha,\beta) \in F, \ d(\alpha) \neq
d(\beta)} r_{\alpha,\beta} s_\alpha s_{\beta^*},$$
where in particular $r_{\delta,\epsilon} \neq 0$. Let $b := s_{\delta^*} a s_\epsilon \neq 0$. Notice that $b \in {\rm Ker}\ \Phi$  because ${\rm Ker}\ \Phi$ is an ideal of
${\rm KP}_R(\L)$. Now by Remark \ref{RemarkRegularPaths} (c), fix $x \in Z(s(\delta)) \cap \mathfrak{T}_\L$. For each $(\alpha,\beta) \in F$ we have two
possibilities:
\begin{itemize}
\item If $x \notin F_{\alpha, \beta}$, then there exists $\mu = \mu_{\alpha,\beta}, \nu = \nu_{\alpha,\beta} \in \L$ as in Lemma \ref{ReductionLemma}(a)
    such that $x \in Z(\mu) \cap Z(\nu)$ and
$$ s_\mu s_{\mu^*} s_{\alpha} s_{\beta^*} s_\nu s_{\nu^*} = 0.$$

\item Otherwise, if $x \in F_{\alpha,\beta}$ then there exists $\gamma = \gamma_{\alpha,\beta} \in \L$ as in Lemma \ref{ReductionLemma}(b) such that $x
    \in Z(\alpha\gamma) \cap Z(\beta\gamma)$ and
$$ s_{\alpha\gamma} s_{(\alpha\gamma)^*} s_{\alpha} s_{\beta^*} s_{\beta\gamma} s_{(\beta\gamma)^*} =  s_{\alpha\gamma} s_{(\beta\gamma)^*} \text{ with }
(\alpha\gamma, \beta\gamma) \text{ cycline}.$$
\end{itemize}

Now let $m$ be the following product:
$$( \prod_{(\alpha,\beta) \in F, \ x \in F_{\alpha, \beta}}s_{\alpha\gamma} s_{(\alpha\gamma)^*})(\prod_{(\alpha,\beta) \in F, \ x \notin F_{\alpha,
\beta}}s_\mu s_{\mu^*}) \ b \ (\prod_{(\alpha,\beta) \in F, \ x \notin F_{\alpha, \beta}}s_\nu s_{\nu^*})(\prod_{(\alpha,\beta) \in F, \ x \in F_{\alpha,
\beta}}s_{\beta\gamma} s_{(\beta\gamma)^*})$$

Observe that every $s_{\alpha\gamma} s_{(\alpha\gamma)^*}, s_{\beta\gamma} s_{(\beta\gamma)^*}, s_\mu s_{\mu^*}, s_\nu s_{\nu^*}$ is inside $\mathcal{D}$.
We have that $m \in \mathcal{M}$ by construction. Since $b \in {\rm Ker}\ \Phi$, we have $m \in {\rm Ker}\ \Phi$.

Let us see that $m \neq 0$: we show that the zero-graded component of $m$ is non-zero, that is, $m_0 \neq 0$. We have
$$m= (\prod s_{\alpha\gamma} s_{(\alpha\gamma)^*})(\prod s_\mu s_{\mu^*}) \ r_{\delta,\epsilon} p_{s(\delta)} \ (\prod s_\nu s_{\nu^*})(\prod
s_{\beta\gamma} s_{(\beta\gamma)^*}) \ + \  m',$$
where $$m' =(\prod s_{\alpha\gamma} s_{(\alpha\gamma)^*})(\prod s_\mu s_{\mu^*}) \ (\sum_{d(\alpha) \neq d(\beta)} r_{\alpha,\beta} s_\alpha s_{\beta^*}) \
(\prod s_\nu s_{\nu^*})(\prod s_{\beta\gamma} s_{(\beta\gamma)^*}).$$
We claim that
$$m_0 = (\prod s_{\alpha\gamma} s_{(\alpha\gamma)^*})(\prod s_\mu s_{\mu^*}) \ r_{\delta,\epsilon} p_{s(\delta)} \ (\prod s_\nu s_{\nu^*})(\prod
s_{\beta\gamma} s_{(\beta\gamma)^*}) \in \mathcal{D}.$$

To see this, consider $(\prod s_{\alpha\gamma} s_{(\alpha\gamma)^*})(\prod s_\mu s_{\mu^*}) \ r_{\alpha,\beta} s_\alpha s_{\beta^*} \ (\prod s_\nu
s_{\nu^*})(\prod s_{\beta\gamma} s_{(\beta\gamma)^*})$ a term in $m'$. Since $ \prod s_{\alpha\gamma} s_{(\alpha\gamma)^*}, \prod s_\mu s_{\mu^*}, \prod
s_\nu s_{\nu^*}, \prod s_{\beta\gamma} s_{(\beta\gamma)^*}$ are each $0$-graded,
$$(\prod s_{\alpha\gamma} s_{(\alpha\gamma)^*})(\prod s_\mu s_{\mu^*}) \ r_{\alpha,\beta} s_\alpha s_{\beta^*} \ (\prod s_\nu s_{\nu^*})(\prod
s_{\beta\gamma} s_{(\beta\gamma)^*})$$
is $d(\alpha)-d(\beta) \neq 0$ graded. Thus $m_0$ is as claimed.

Recall from Lemma \ref{DiagonalIsomorphismLemma}, we have an isomorphism $\pi: \mathcal{D}
\rightarrow A_{\mathcal{D}}$. To see that $m_0 \neq 0$, it suffices to show $\pi(m_0) \neq 0$. Now
\begin{align*}
\pi(m_0) &=  r_{\delta,\epsilon} (\prod 1_{Z(\alpha\gamma)}) (\prod 1_{Z(\mu)} \ 1_{Z(s(\delta))}) \ (\prod 1_{Z(\nu)}) (\prod 1_{Z(\beta\gamma)}) \\
&= r_{\delta,\epsilon} 1_{U}
\end{align*}
where $$U= \left(\bigcap Z(\alpha\gamma)\right) \cap \left(\bigcap Z(\mu)\right) \cap Z(s(\delta)) \cap \left(\bigcap Z(\nu)\right) \cap \left(\bigcap
Z(\beta\gamma)\right).$$
By construction $x \in U$ and hence $1_U \neq 0$. So $\pi(m_0) \neq 0$ and finally $m \neq 0$ as desired.

This contradicts the assumption that $\Phi |_{\mathcal{M}}$ is injective because we found $0 \neq m \in \mathcal{M}$ and $m \in {\rm Ker}\ \Phi$. Therefore
$\Phi$ is injective.
\end{proof}

From Theorem \ref{UniquenessTh} we can recover the usual Cuntz-Krieger uniqueness theorem given in \cite[Theorem 4.7]{ACaHR}. First we need to consider the following
lemma which is proved similarly to \cite[Proposition 4.8]{BNR}.

\begin{lemma}\label{lem:aperiodic}
If $\L$ is aperiodic, then $(\alpha,\beta)$ is a cycline pair if and only if $\alpha = \beta$. In particular, $\mathcal{M}=\mathcal{D}$.
\end{lemma}

\begin{corollary} Let $\L$ be an aperiodic row-finite $k$-graph without sources and $R$ be a commutative ring with $1$. If $\Phi: {\rm KP}_R(\Lambda)
\rightarrow A$ is a ring homomorphism, then $\Phi$ is injective if and only if $\Phi(rp_v)\neq 0$ for all $r \in R \setminus \{0\}$ and $v \in \L^0$.
\end{corollary}

\begin{proof} We verify that $\Phi(rp_v)\neq 0$ for all $r \in R \setminus \{0\}$ and $v \in \L^0$ implies $\Phi$ is injective. We show $\Phi |_{\mathcal{M}}$ is injective, which suffices by Theorem \ref{UniquenessTh}. First, by Lemma \ref{lem:aperiodic}, $\mathcal{M}=\mathcal{D}$. By way of contradiction suppose there exist $r \in R \setminus \{0\}$ and $\lambda \in \L$ such that $\Phi(r s_\lambda s_{\lambda^*}) = 0$. But then
$\Phi(s_{\lambda^*}(r s_\lambda s_{\lambda^*})s_{\lambda}) = 0$. Thus $\Phi(r p_{s(\lambda)})=0$, which is a contradiction.
\end{proof}

\section*{acknowledgements}

The first author is supported by the Marsden grant 15-UOO-071 from the Royal Society of New Zealand and a University of Otago Research Grant.
A portion of this work was carried out while she was visiting the Universidad de M\'{a}laga.
She would like to thank her hosts for their support.

The second author was partially supported by the Spanish MEC and Fondos FEDER through project MTM2013-41208-P, and by the
Junta de Andaluc\'{\i}a and Fondos FEDER, jointly, through project FQM-7156. Part of this work was carried out
during a visit of the second author to the Institute for
Research in Fundamental Sciences (IPM-Isfahan) in Isfahan, Iran. The second author thanks this
host institution for its warm hospitality and support.

The research of the
third author was in part supported by a grant from IPM (No. 94170419).


\end{document}